\documentclass[12pt]{article}
\usepackage[top=1in, bottom=1in, left=1in, right=1in]{geometry}
\usepackage{graphicx, amsfonts,amssymb,amsthm,amsbsy,latexsym,amscd,amsmath,dsfont,euscript,enumerate,manfnt, marvosym,verbatim, calc,mathrsfs,marvosym,textcomp, color,xcolor,tikz-cd,setspace,stmaryrd,upgreek,bm} 

\usepackage{pgf,tikz,amsmath}
\usepackage{blkarray}
\usepackage{caption}
\usepackage{enumerate}
\usepackage[pagebackref,colorlinks=true,
linkcolor=blue,
urlcolor=red,colorlinks,
citecolor=green]{hyperref}
\usepackage{cleveref}
\usepackage{parskip}
\usepackage{setspace}
\makeatletter
\def\thm@space@setup{%
  \thm@preskip=10pt 
  \thm@postskip=6pt 
}
\makeatother

\usepackage{refcount}

\newtheorem*{theorem*}{Theorem}
\newtheorem{theorem}{Theorem}

\newtheorem{lemma}{Lemma}[section]

\crefname{lemma}{Lemma}{Lemmas}
\crefname{corollary}{Corollary}{Corollaries}

\Crefname{lemma}{Lemma}{Lemmas}
\Crefname{corollary}{Corollary}{Corollaries}

\crefalias{corollary}{corollary}

\newtheorem{claim}{Claim}[theorem]

\crefname{claim}{claim}{claims}
\Crefname{claim}{Claim}{Claims}

\crefname{subclaim}{subclaim}{subclaims}
\Crefname{subclaim}{Subclaim}{Subclaims}

\theoremstyle{definition}
\newtheorem*{definition}{Definition}

\newtheorem*{example}{Example}
\newtheorem*{acknowledgement}{Acknowledgement}

\newcommand{\OO}{\mathcal O}

\newcommand{\E}{\mathbb E}

\newcommand{\Q}{\mathbb Q}

\newcommand{\1}{\mathbf 1}

\newcommand{\Ind}{\operatorname{Ind}}

\newcommand{\Fix}{\operatorname{Fix}}
\newcommand{\Cay}{\operatorname{Cay}}

\title{Even-Intersecting Families of Permutations}
\author{
    Anirban Banerjee\thanks{\text{anirban.banerjee@iiserkol.ac.in}} \qquad
    Abisek Dewan\thanks{\text{ad22rs069@iiserkol.ac.in}} \qquad
    Rajiv Mishra\thanks{\text{rm20rs017@iiserkol.ac.in}} \\[0.8em]
    \small Department of Mathematics and Statistics, IISER Kolkata, India
}
\date{ }

\begin{document}

\maketitle

\begin{abstract}
    A family of permutations in $S_n$ is called even-intersecting if every two distinct members agree in an even number of positions. Let $M(n)$ denote the maximum size of such a family. For even $n$, we prove that
    $$n!!\leq M(n)\leq e^{\frac{n}{2}+o(n)}n!!,$$
    improving the bound obtained from a theorem of Cameron, Deza and Frankl (1987) by an exponential factor. 
    This problem may be viewed as a permutation analogue of the classical Eventown problem for set systems. For odd $n$, we give a construction yielding $M(n)\geq n^2/4$. We further extend this construction to obtain
    $$ M(n)\geq\bigl(n-\sqrt{n-1}\bigr)^2 ,$$
   whenever $n=(q+1)^2+1$ and $q$ is an odd prime power. The latter bound asymptotically matches the upper bound $M(n)\leq(n-1)^2+1$ obtained by  Cameron, Deza and Frankl.
\end{abstract}



\section{Introduction}

 Intersection problems are a central theme in extremal combinatorics. Given a collection of combinatorial objects equipped with a suitable notion of intersection, one seeks to determine the largest family satisfying a prescribed intersection condition. The classical example is the Erd\H{o}s--Ko--Rado theorem. A family $\mathcal{F}\subseteq\binom{[n]}{k}$ is called intersecting if
\[
A\cap B\neq\varnothing
\qquad\text{for every } A,B\in\mathcal{F}.
\]
The Erd\H{o}s--Ko--Rado theorem \cite{erdos1961intersection} states that, for $n\geq 2k$,
\[
|\mathcal{F}|\leq \binom{n-1}{k-1},
\]
with equality, for $n>2k$, precisely for a star  i.e, a family of all $k$-sets containing a fixed element. This theorem has motivated a large body of work on intersection problems for sets and for many other combinatorial objects; see, for example, the monograph of Gerbner and Patk\'os~\ \cite{gerbner2018extremal}, Intersection Problems in Extremal Combinatorics, a survey by David Ellis \cite{ellis2022intersection}, Erdõs–Ko–Rado Theorems: Algebraic Approaches, a book by Godsil and Meagher \cite{godsil1945erdos}.

One natural setting for such questions is the symmetric group $S_n$. For $\sigma,\pi\in S_n$, define their intersection
\[
\sigma \cap \pi
=
\{i\in[n]:\sigma(i)=\pi(i)\},
\]
the number of positions at which $\sigma$ and $\pi$ agree or equivalently the set of fixed points of $\sigma \pi^{-1}$. A family $\mathcal{F}\subseteq S_n$ is called intersecting if
$|\sigma \cap \pi| \ge 1$ for every $\sigma,\pi \in \mathcal{F}$.
Deza and Frankl~\cite{frankl1977maximum} proved that every intersecting family $\mathcal{F}\subseteq S_n$ has size at most $(n-1)!$, with equality attained by the family
\[
\{\sigma\in S_n:\sigma(i)=j\}
\]
for fixed $i,j\in[n]$. Cameron and Ku~\cite{cameron2003intersecting}, independently of Larose and Malvenuto~\cite{larose2004stable}, characterized the extremal families. More generally, the study of $t$-intersecting families of permutations, in which every two permutations agree in at least $t$ positions, was initiated by Deza and Frankl and subsequently developed by Ellis, Friedgut and Pilpel~\cite{ellis2011intersecting} and by subsequent work of Keller, Lifshitz, Minzer and Sheinfeld~\cite{keller2024t}. These results establish a rich EKR-type theory for permutations.

Our interest, however, is in a different type of intersection condition: rather than requiring two objects to have a positive number of common elements, we require their intersection to have prescribed parity.
A natural example of a parity-based intersection problem is the classical \emph{Eventown problem}  on sets. A family $\mathcal{F}\subseteq 2^{[n]}$ is called an \emph{Eventown family} if $|A\cap B|\equiv0\pmod 2$ for all $A,B\in\mathcal{F}$. In particular, every member of $\mathcal{F}$ has even cardinality. Answering a question of Erd\H{o}s, Berlekamp~\cite{berlekamp1969subsets} proved that $|\mathcal{F}|\leq2^{\lfloor n/2\rfloor}$, and the bound is attained by the \emph{atomic construction}: partition $2\lfloor n/2\rfloor$ elements of $[n]$ into pairs and take $\mathcal{F}$ to be the family of all unions of these pairs. The same bound was independently obtained by Graver~\cite{graver1975boolean}. The optimality of the atomic construction has a particularly elegant linear-algebraic interpretation which can be found in Babai and Frankl's \emph{Linear Algebra Methods in Combinatorics}~\cite{babai1992linear}.

Here, we investigate the corresponding problem for permutations. 
We call a family $\mathcal{F}\subseteq S_n$ \emph{even-intersecting} if
\[
|\sigma \cap \pi|\equiv0\pmod 2
\qquad\text{for every distinct } \sigma,\pi\in\mathcal{F}.
\]
Equivalently,
\[
\big|\Fix(\sigma\pi^{-1})\big|
\equiv0\pmod2
\qquad\text{for every distinct } \sigma,\pi\in\mathcal{F}.
\]
When $n$ is even the condition $|\sigma \cap \pi| \equiv 0\pmod2$ holds for all $\sigma,\pi \in \mathcal{F}$. Thus, for even $n$, our problem may be viewed as an Eventown problem in which subsets are replaced by permutations and set intersection is replaced by agreement in positions.
Define 
$$M(n):=\max\left\{|\mathcal{F}|:\mathcal{F}\subseteq S_n \text{ is even-intersecting}\right\}.$$
We are interested in the upper and lower bound of $M(n)$.

In 1987, Cameron, Deza and Frankl \cite{cameron1987sharp} studied a more general intersection
problem for permutations. For a set
$L\subseteq\{0,1,\ldots,n-1\}$, a family $\mathcal{F}\subseteq S_n$ is called
$L$-intersecting if, for every two distinct $\sigma,\tau\in\mathcal{F}$, the
size of their intersection, $|\sigma\cap \tau|$, belongs to $L$. They
denoted the maximum size of such a family by $m(n,L)$. Thus 
the even-intersecting problem corresponds to $
L=\{0, 2, 4,\ldots\},
$ and hence $M(n)=m(n,L)$, in this case.

 Cameron, Deza and Frankl  \cite{cameron1987sharp} proved that if $p$ is a prime such that no
element of $L$ is congruent to $n$ modulo $p$, and if $L$ is contained in $s$
residue classes modulo $p$, then
$$
m(n,L)
\leq
\sum_{\substack{\lambda\vdash n\\ n-\lambda_1\leq s}}
(f^\lambda)^2.
$$
For $L=\{0,2,\ldots,n-2\}$ and even $n$, taking a prime
$p>n$ gives $s=n/2$. Applying their theorem and evaluating
the resulting character sum up to polynomial factors (see the \Cref{lem:CDF-sum-asymptotic} in the appendix \ref{app:CDF-asymptotic}), we obtain
\begin{equation}
M(n)
\leq
\sum_{\substack{\lambda\vdash n\\ \lambda_1\geq n/2}}
(f^\lambda)^2= 8^{{n}/{2}+O(\log n)}n!!.
\label{eq:CDF-even-bound}
\end{equation}

Cameron, Deza and Frankl also posed a more specific problem concerning
intersection sizes divisible by a fixed integer. In
Problem~5.12 of \cite{cameron1987sharp}, they asked whether, for sufficiently large
$r$, every family of permutations on $rm$ points whose pairwise intersection
sizes belong to $
\{0,m,2m,\ldots,(r-1)m\}
$ has size at most $r!\,m^r.$

Taking $m=2$ and writing $n=2r$, this specializes to the even-intersecting
problem and predicts the bound
$$
M(n)\leq n!!
$$
for all sufficiently large even $n$. 

For even $n$, we obtain an upper bound for $M(n)$ that is exponentially stronger than the Cameron--Deza--Frankl bound in \Cref{eq:CDF-even-bound}. More precisely, we prove the following.

\begin{theorem}\label{thm:main_result}
Let $M(n)$ be the maximum size of an even-intersecting family of $S_n$.
    Then, there exist absolute constants
$C>0$ and $n_0 \in \mathbb{N}$ such that, for every even $n\geq n_0$,
\begin{equation*}
n!!\leq M(n)
\leq e^{{n}/{2}+C n^{3/4}\log n}\,n!!,
\label{eq:main-even-intro}
\end{equation*}
where the lower bound is attained by $\mathcal{F}=S_2\wr S_{n/2}$.
\end{theorem}

Consequently, the bound in \Cref{thm:main_result} improves their bound given in \Cref{eq:CDF-even-bound} by an exponential
factor of order $\left({8}/{e}\right)^{\,n/2-o(n)}$.

There is a closely related recent work of Lindzey \cite{lindzey2026eventown}, who studied an Eventown-type problem for permutations under a different notion of intersection. In his work, a family is called \emph{even-cycle-intersecting} if $\sigma\pi^{-1}$ contains an even cycle for every pair $\sigma,\pi$ in the family. Although this condition is related to parity phenomena in the symmetric group, it is different from the agreement-based condition considered here. In particular, our condition concerns the parity of the number of fixed points of $\sigma\pi^{-1}$, whereas Lindzey's condition concerns its cycle structure. 

For odd $n$, an even-intersecting family is $L$-intersecting with $L=\{0,2,4,\ldots, n-3\}$. Taking  $p=2$, in Cameron, Deza, Frankl result \cite{cameron1987sharp},  $L$ is covered by just one residue class modulo 2, hence $s=1$. Therefore, we have
$$M(n)\leq \sum_{\substack{\lambda\vdash n\\ n-\lambda_1\leq 1}}
(f^\lambda)^2= \left(f^{(n-1,1)}\right)^2+\left(f^{(n)}\right)^2=(n-1)^2+1.$$

For odd $n$, we  establish lower bounds for $M(n)$ by giving two explicit constructions of even-intersecting families. The first applies to every odd $n$, while the second gives a substantially larger family when $n=(q+1)^2+1$ for some odd prime power $q$.

\begin{theorem}\label{Thm:loose lower bound}
For every odd $n$,
$$
M(n)\geq
\begin{cases}
\displaystyle \frac{(n+1)^2}{4},
& \text{if } n\equiv 1 \pmod 4,\\[2mm]
\displaystyle \frac{(n+1)^2}{4}-1,
& \text{if } n\equiv 3 \pmod 4.
\end{cases}
$$
\end{theorem}

\begin{theorem}\label{Thm:tight lower bound}
If $n=(q+1)^2+1$ for some odd prime power $q$, then
$$
M(n)\geq (n-\sqrt{n-1})^2=n^2-O(n^{3/2}).
$$
\end{theorem}
\subsection{Structure of the paper}
The remainder of the paper is organized as follows. In \Cref{sec:preliminaries}, we collect the preliminary results from representation theory, symmetric functions, and spectral graph theory that will be used throughout the paper. In \Cref{sec:optimization_bound}, we develop the technical lemmas and the character-theoretic optimization bound needed for the main result. \Cref{sec:Even_intersecting_family} is devoted to even-intersecting families of permutations: in \Cref{subsec:Even_intersecting_family_even_n}, we treat the case of even $n$ and prove Theorem 1 using a weighted Cayley graph and Hoffman's bound, while in \Cref{subsec:Even_intersecting_family_odd_n}, we consider odd $n$ and give two constructions yielding lower bounds for $M(n)$.
\section{Preliminaries} 
\label{sec:preliminaries}
 In this section, we present some results from the representation theory of symmetric groups and the ring of symmetric functions that will be used in this paper. We state these results without proof; interested readers may consult standard references, such as \cite{macdonald1998symmetric,sagan2001symmetric,Stanley_2023}, for further details.

\subsection{Representation Theory of Symmetric Groups}

A \emph{partition} of $n \in \mathbb{N}$, denoted $\lambda \vdash n$, is a sequence of positive integers $\lambda = (\lambda_1, \lambda_2, \dots, \lambda_r)$ such that $\lambda_1 \ge \lambda_2 \ge \dots \ge \lambda_r > 0$ and $\sum_{i=1}^r \lambda_i = n$. Let $m_i=m_i(\lambda)$ denote the number of parts equal to $i$ in $\lambda$, then we can denote $\lambda$ as $\lambda=(1^{m_1}2^{m_2}\ldots)$. The conjugacy classes of $S_n$ are uniquely indexed by partitions of $n$ corresponding to cycle types. For a symmetric group the set of distinct irreducible representation is in one-to-one correspondence with the set of partitions of $n$. We denote the irreducible representation of $S_n$ corresponding to the  partition $\lambda \vdash n$ by $S^{\lambda}$ (often called specht module). The character of $S^{\lambda}$ is denoted by $\chi^{\lambda}$. The character value on the identity element gives the dimension of the representation, denoted by $f^{\lambda}$. 

There is an efficient way to calculate the dimension of an irreducible representation of $S_n$ using the hook length formula. A Young diagram $Y_\lambda$ associated with a partition
$\lambda=(\lambda_1,\lambda_2,\ldots,\lambda_r)\vdash n$ is a left-justified array of cells with $\lambda_i$ cells in the $i$-th row. For a cell $u=(i,j)\in Y_\lambda$, define its hook as the set
$$H(u)=\{(i,j')\in Y_\lambda:j'\geq j\}
\cup
\{(i',j)\in Y_\lambda:i'>i\}.$$
In other words, the hook of $u$ consists of the cell $u$, the cells to its right in the same row, and the cells below it in the same column. The hook length $h(u)$ is the cardinality of $H(u)$.

\begin{theorem}[Hook Length Formula \cite{frame1954hook}]
Let $\lambda\vdash n$ and let $Y_\lambda$ be the Young diagram corresponding to $\lambda$. Then the dimension $f^\lambda$ of the irreducible representation $S^{\lambda}$ of $S_n$ is given by
$$f^\lambda
=
\frac{n!}{\displaystyle\prod_{u\in Y_\lambda} h(u)},$$
where $h(u)$ denotes the hook length of the cell $u$.
\end{theorem}


\subsection{The Ring of Symmetric Functions}

Let $\mathbf{x}=\{x_1,x_2,\ldots\}$ be an infinite set of variables. For a
partition $\lambda=(\lambda_1,\lambda_2,\ldots,\lambda_r)$, the
\emph{monomial symmetric function} corresponding to $\lambda$ is defined by
$$
m_\lambda=m_\lambda(\mathbf{x})
:=
\sum_{\alpha}x_1^{\alpha_1}x_2^{\alpha_2}\cdots,
$$ where the sum is over all distinct rearrangements $\alpha$ of
$(\lambda_1,\lambda_2,\ldots,\lambda_r,0,0,\ldots)$.

The \emph{ring of symmetric functions} over $\mathbb{Q}$ is the space spanned by all $m_\lambda$, that is, 
$$
\Lambda
=
\operatorname{span}_{\mathbb{Q}}
\{m_\lambda:\lambda \text{ is a partition}\}.
$$
It is a graded ring with decomposition
$$
\Lambda
=
\bigoplus_{n\geq 0}\Lambda^n,
$$
where $\Lambda^n=\operatorname{span}_{\mathbb{Q}}
\{m_\lambda:\lambda\vdash n\}.
$ In particular, $\{m_\lambda:\lambda\vdash n\}$ is a basis of
$\Lambda^n$.

Besides the monomial symmetric functions, there are several other standard
families of symmetric functions which give natural bases of $\Lambda^n$.
We first introduce the power-sum and complete homogeneous symmetric
functions.

For a positive integer $m$, the \emph{power-sum symmetric function} of
degree $m$ is defined by
$$
p_m=p_m(\mathbf{x})
:=
\sum_{i\geq 1}x_i^m.
$$
The \emph{complete homogeneous symmetric function} of degree $m$ is defined
by
$$
h_m=h_m(\mathbf{x})
:=
\sum_{i_1\leq i_2\leq\cdots\leq i_m}
x_{i_1}x_{i_2}\cdots x_{i_m},
$$
with the convention $h_0=1$. For a partition
$\lambda=(\lambda_1,\lambda_2,\ldots,\lambda_r)$, define
$$
p_\lambda
:=
p_{\lambda_1}p_{\lambda_2}\cdots p_{\lambda_r},
\qquad
h_\lambda
:=
h_{\lambda_1}h_{\lambda_2}\cdots h_{\lambda_r}.
$$
Both $\{p_\lambda:\lambda\vdash n\}$ and $\{h_\lambda:\lambda\vdash n\}
$ form bases of $\Lambda^n$. The power-sum and complete homogeneous symmetric functions are related by the generating function
\begin{equation*}
    \sum_{m\geq 0}h_mu^m
=
\exp\left(
\sum_{k\geq 1}p_k\frac{u^k}{k}
\right).
\end{equation*}

Another fundamental basis of $\Lambda$ is given by the \emph{Schur symmetric
functions}. Let $\lambda$ be a partition. A \emph{semistandard Young tableau}
of shape $\lambda$ is a filling of the Young diagram of $\lambda$ with
positive integers such that the entries are weakly increasing along each
row and strictly increasing down each column. If $T$ is a semistandard
Young tableau, let $m_i(T)$ denote the number of entries equal to $i$, and
write
$$
\mathbf{x}^T
=
x_1^{m_1(T)}x_2^{m_2(T)}\cdots.
$$
The \emph{Schur symmetric function} corresponding to $\lambda$ is defined by
$$
s_\lambda=s_\lambda(\mathbf{x})
:=
\sum_T \mathbf{x}^T,
$$
where the sum is over all semistandard Young tableaux $T$ of shape
$\lambda$. The set $
\{s_\lambda:\lambda\vdash n\}
$ forms a basis of $\Lambda^n$. For a partition $\rho=(1^{m_1}2^{m_2}\cdots)$, define
$$
z_\rho=\prod_{i\geq 1}i^{m_i}m_i!.
$$
The \emph{Hall inner product} on $\Lambda$ is defined on the power-sum basis by
$$
\langle p_\lambda,p_\mu\rangle
=
\delta_{\lambda,\mu}z_\lambda.
$$
With respect to this inner product, the Schur functions form an
orthonormal basis; that is,
$$
\langle s_\lambda,s_\mu\rangle
=
\delta_{\lambda,\mu}.
$$

We now recall the Frobenius characteristic correspondence between
symmetric functions and characters of symmetric groups. For each $n\geq 0$, let $R^n$ denote the space of
class functions on $S_n$. 
The \emph{Frobenius characteristic map} $
\operatorname{ch}^n:R^n\longrightarrow\Lambda^n$ is defined by
$$
\operatorname{ch}^n(f)
=
\sum_{\rho\vdash n}
\frac{f(\rho)}{z_\rho}p_\rho,
$$
where $f(\rho)$ denotes the value of $f$ on the conjugacy class of
cycle type $\rho$.

\begin{lemma}[Frobenius characteristic correspondence]
\label{lem:Frobenius_characteristic}
For every
$\lambda\vdash n$,
$$
\operatorname{ch}^n(\chi^\lambda)=s_\lambda,
$$
where $\chi^\lambda$ denotes the irreducible character of $S_n$
indexed by $\lambda$.
\end{lemma}

For $\lambda,\rho\vdash n$, the Frobenius characteristic correspondence gives $\langle p_\rho,s_\lambda\rangle=\chi^\lambda(\rho).$ Equivalently,
\begin{equation}\label{eq:p_1_j}
p_\rho
=
\sum_{\lambda\vdash n}
\chi^\lambda(\rho)s_\lambda.
\end{equation}

Now let $
R=\bigoplus_{n\geq 0}R^n.
$ The maps $\operatorname{ch}^n$ combine to give a graded isomorphism
$$
\operatorname{ch}
=
\bigoplus_{n\geq 0}\operatorname{ch}^n
:
R\longrightarrow\Lambda.
$$
For $f\in R^m$ and $g\in R^n$, define
$$
f\circ g
=
\Ind_{S_m\times S_n}^{S_{m+n}}
(f\otimes g).
$$
With this product, $R$ is a graded algebra, and $\operatorname{ch}(f\circ g)=\operatorname{ch}(f)\operatorname{ch}(g).$ Thus, $\operatorname{ch}:R\longrightarrow\Lambda$ is an isomorphism
of graded algebras.


We next recall Pieri's rule, which describes the product of a Schur
function with a complete homogeneous symmetric function. For partitions
$\mu\subseteq\lambda$, the skew diagram $\lambda/\mu$ is called a
\emph{horizontal $r$-strip} if it consists of $r$ boxes, with no two
boxes in the same column.

\begin{lemma}[Pieri's Rule]
\label{lem:Pieri's_Formula}
Let $\mu$ be a partition and let $r$ be a positive integer. Then
$$
s_\mu h_r
=
\sum_{\lambda}s_\lambda,
$$
where the sum is over all partitions $\lambda$ such that
$\lambda/\mu$ is a horizontal $r$-strip.
\end{lemma}

Another useful relation between the Schur and complete homogeneous
symmetric functions is given by the Jacobi--Trudi identity.

\begin{lemma}[Jacobi--Trudi Identity]
\label{lem:Jacobi_Trudi}
For a partition $\lambda=(\lambda_1,\lambda_2,\ldots,\lambda_\ell)$,
$$
s_\lambda
=
\det\left(h_{\lambda_i-i+j}\right)_{1\leq i,j\leq\ell},
$$
where $h_0=1$ and $h_r=0$ for $r<0$.
\end{lemma}


\subsection{Cayley Graphs and the Hoffman Bound}
Let $G$ be a finite group and let $X\subseteq G \backslash \{e\} $ be a generating set such that it is inverse-closed that is
$X=X^{-1}$. The \textit{Cayley graph}
$\operatorname{Cay}(G,X)$ is the graph whose vertex set is $G$, with two
vertices $g,h\in G$ adjacent whenever $g^{-1}h\in X$ .
A Cayley graph is called a normal if the generating set
$S$ is a union of conjugacy classes of $G$. 
 Let
$O_n=\{\sigma\in S_n:|\Fix(\sigma)|\   \text{is odd}\}$. We shall study the Cayley graph Cay$(G,O_n)$ on the symmetric group $S_n$ when $n$ is even. Clearly Cay$(S_n,O_n)$ is normal.

The eigenvalues of a normal Cayley graph can be expressed in terms of the
irreducible characters of the underlying group. The following result which is commonly credited to Babai \cite{babai1979spectra}, or Diaconis and Shahshahani \cite{diaconis1981generating} gives
the formula for these eigenvalues.
\begin{theorem}[Babai \cite{babai1979spectra}, Diaconis and Shahshahani \cite{diaconis1981generating}]\label{thm:Eigenvalue_chatacter_formula}
Let $G$ be a finite group and $X \subseteq G$ be a union of conjugacy classes and is inverse-closed. For the normal Cayley graph
$\operatorname{Cay}(G,X)$, the eigenvalues of its adjacency matrix are
\[
\theta_\rho
=
\frac{1}{\dim(\rho)}
\sum_{\tau\in X}\chi^\rho(\tau),
\]
where $\rho$ ranges over the irreducible representations of $G$ and
$\chi^\rho$ is its character. The eigenvalue $\theta_\rho$ has
multiplicity $\dim(\rho)^2$.
\end{theorem}

The character formula also extends naturally to weighted versions of normal Cayley graphs. One way to create a weighted adjacency matrix for a Cayley graph is as follows. Let Cay$(S_n,X)$ be a normal Cayley graph and $X_1,X_2,...,X_t$ be the conjugacy classes that are contained in $X$ and let $A_i$ denote the adjacency matrix of the Cayley graph Cay$(S_n,X_i)$ . Choose real numbers $c_1,c_2,...,c_t$ and define the matrix $A=\sum{c_iA_i}$. Then $A$ is a weighted adjacency matrix for the Cayley graph Cay$(S_n,X)$ with constant row sum $\sum{c_i |X_i|}$. Since the weights are constant on conjugacy classes, therefore by  \Cref{thm:Eigenvalue_chatacter_formula}, the eigenvalues of the weighted adjacency matrix $A$ of the Cayley graph Cay$(S_n,X)$ can be expressed directly in terms of the irreducible characters of \(S_n\).
\begin{lemma}\label{lem:Weighted_Adjacency_eigenvlues}
 Let $X_1,X_2,...,X_t$ be conjugacy classes of $S_n$. For real numbers $c_1,c_2,...,c_t$ define $A=\sum_{i=1}^{t}c_i A_i$, where $A_i$ is the adjacency matrix of the normal Cayley graph $\Cay(S_n,X_i)$. If $g_i \in X_i$, then the eigenvalue of $A$ corresponding to irreducible character $\chi^{\rho}$ is given by
    \[\theta_{\rho}=\frac{1}{\dim(\rho)} \sum_{i=1}^{t}c_i|X_i| \chi^{\rho}(g_i)\]
\end{lemma}
whose multiplicity is $\text{dim}(\rho)^2$.

Having determined the eigenvalues of the weighted adjacency matrix $A$, we now use this spectral information to bound the size of an independent set. The precise result that we shall use for this purpose is Hoffman’s ratio bound, stated below in a form suitable for weighted adjacency matrix of a graph. For a proof see \cite{ellis2011intersecting,meagher2021erdHos}.

\begin{theorem}[Weighted Hoffman's ratio bound]
Let $G$ be a connected graph with vertex set $V(G)$. Let $A$ be a weighted adjacency matrix of $G$ with constant row sum $d$. Then,
\[\alpha(G) \le |V(G)| \frac{-\lambda_{min}}{d-\lambda_{min}}\]
    where $\lambda_{min}$ denotes the least eigenvalue of $A$. Furthurmore, if equality holds for some independent set $I \subseteq V(G)$ whose characteristic vector is $v_I$, then the balanced characteristic vector $v_I-\frac{|I|}{|V(G)|}\textbf{1}$ belongs to the eigenspace corresponding to $\lambda_{min}$.
\end{theorem}

We shall also need a standard generalization of the Vandermonde determinant, in which the monomials $1,x,\ldots,x^{n-1}$ are replaced by polynomials of the corresponding degrees.

\begin{lemma}[General Vandermonde Determinant\cite{krattenthaler1999advanced}]\label{lem:General_Vandermonde} Let $P_1, P_2, \ldots, P_n$ are the polynomials of the form $P_j(x)=a_jx^{j-1}+$ lower terms, then $$\det[P_j(x_i)]_{1\leq i,j\leq n}=a_1a_2\ldots a_n \prod_{1\leq i<j\leq n}(x_j-x_i).$$ \end{lemma}
\section{A Character-Theoretic Optimization Bound}
\label{sec:optimization_bound}

\subsection{A Fixed-Point Transform}

For a given partition $\lambda=(\lambda_1,\lambda_2,\dots,\lambda_r)\vdash n$, we introduce the following notation used throughout this section:
$$
\ell:=\lambda_1,\qquad d:=n-\ell,\qquad
\overline{\lambda}:=(\lambda_2,\dots,\lambda_r)\vdash d.
$$
We shall use the standard cycle-index identity
\begin{equation}
\sum_{m\geq 0}u^m
\sum_{\rho\vdash m}
\frac{1}{z_\rho}
\prod_{k\geq 1}t_k^{m_k(\rho)}
=
\exp\left(
\sum_{k\geq 1}t_k\frac{u^k}{k}
\right).
\label{eq:cycle-index}
\end{equation}

This is the usual cycle-index identity for the symmetric groups, see \cite[Example: 5.2.10]{Stanley_2023}.

Further, for $m\in\mathbb{Z}$ and $t\in\mathbb{R}$, we define the truncated exponential polynomial as
$$
E_m(t):=
\begin{cases}
\displaystyle\sum_{j=0}^{m}\frac{t^j}{j!},
& \text{if } m\geq 0,\\[2mm]
0,
& \text{if } m<0.
\end{cases}
$$

We will also need the following asymptotic behavior of the truncated exponential polynomial. The following lemma gives the form that will be useful in our proofs later.

\begin{lemma}
Let $X:=\sqrt{n}+1$ and $Y:=\sqrt{n}-1$. 
As $n\to\infty$, uniformly for all integers $m\geq 4X$,
$$
E_m(Y)=e^Y(1+o(1))
\qquad\text{and}\qquad
E_m(-X)=e^{-X}(1+o(1)).
$$
\end{lemma}

\begin{proof}
Put $k:=m+1$. Since $m\geq 4X$, the alternating-series remainder estimate gives
\begin{align*}
    \left|E_m(-X)-e^{-X}\right|
\leq \frac{X^k}{k!}\qquad\text{and}\qquad
0\leq e^Y-E_m(Y)
\leq  \frac{4}{3}\frac{Y^k}{k!}.
\end{align*}
Using $k!\geq \left(\frac{k}{e}\right)^k$ together with $Y<X$ and $k\geq 4X$, we obtain
$$
\frac{\left|E_m(-X)-e^{-X}\right|}{e^{-X}}
\leq
e^X\left(\frac{eX}{k}\right)^k
\leq
e^X\left(\frac{e}{4}\right)^{4X},
$$
and
$$
\frac{e^Y-E_m(Y)}{e^Y}
\leq
\frac{4}{3}e^{-Y}
\left(\frac{eY}{k}\right)^k
\leq
\frac{4}{3}e^{-Y}
\left(\frac{e}{4}\right)^{4X}.
$$
Both bounds tend to zero as $n\to\infty$, independently of $m\geq 4X$. This proves the result.
\end{proof}

The dependence of character values on the number of fixed points will be a central theme in our work. We therefore introduce a generating function that gives a weighted sum of character values, where the weights depend on the number of fixed points of the corresponding cycle types.

\begin{definition}[Fixed-point transform]
For $\lambda\vdash n$ and $t\in\mathbb{R}$, we define the fixed-point transform by
$$
\Phi_\lambda(t):=
\sum_{\rho\vdash n}
\frac{t^{m_1(\rho)}}{z_\rho}\chi^\lambda(\rho).
$$
\end{definition}

We next derive an expansion of the fixed-point transform that is more suitable for our later calculations. The transform can be expressed in terms of the scalar products of Schur functions with $p_1^j h_{n-j}$.

\begin{lemma}\label{lem:Phi_lambda_series_expansion}
For every $\lambda\vdash n$, we have
\begin{equation}
\Phi_\lambda(1+t)
=
\sum_{j=0}^{n}
\frac{t^j}{j!}\,A_{\lambda,j},
\label{eq:Phi-power-expansion}
\end{equation}
where $A_{\lambda,j}
:=
\left\langle p_1^j h_{n-j},s_\lambda\right\rangle.$
\end{lemma}

\begin{proof}
In the cycle-index identity \eqref{eq:cycle-index}, make the substitutions
$$
t_1\longmapsto (1+t)p_1,
\qquad
t_k\longmapsto p_k
\quad (k\geq 2).
$$
Since $p_\rho=\prod_{k\geq 1}p_k^{m_k(\rho)}$, we obtain
$$
\sum_{m\geq 0}u^m
\sum_{\rho\vdash m}
\frac{(1+t)^{m_1(\rho)}}{z_\rho}\,p_\rho
=
\exp\left(
(1+t)p_1u+
\sum_{k\geq 2}p_k\frac{u^k}{k}
\right).
$$
Using
$$
\sum_{m\geq 0}h_mu^m
=
\exp\left(
\sum_{k\geq 1}p_k\frac{u^k}{k}
\right),
$$
the right-hand side becomes
$$
e^{tp_1u}\sum_{m\geq 0}h_mu^m.
$$
Therefore, taking the homogeneous component of degree $n$, we get
$$
\sum_{\rho\vdash n}
\frac{(1+t)^{m_1(\rho)}}{z_\rho}\,p_\rho
=
\sum_{j=0}^{n}
\frac{t^j}{j!}\,p_1^j h_{n-j}.
$$
Taking the Hall inner product with $s_\lambda$, and using $
\langle p_\rho,s_\lambda\rangle
=
\chi^\lambda(\rho),
$ gives
$$\Phi_\lambda(1+t)=\sum_{j=0}^{n}\frac{t^j}{j!}\left\langle p_1^j h_{n-j},s_\lambda\right\rangle.$$
\end{proof}

The coefficients $A_{\lambda,j}$ appearing in the expansion above can be expressed as sums of the dimensions $f^\nu$, where the partitions $\nu$ are determined by a horizontal-strip condition. The following lemma gives this expression and the bounds on $A_{\lambda,j}$ that will be used later.

\begin{lemma}\label{lem:A_lambda_j bound}
Let $\lambda = (\lambda_1, \lambda_2, \ldots, \lambda_r)\vdash n$. For any integer $j$, define $A_{\lambda,j} = \left\langle p_1^j h_{n-j}, s_\lambda \right\rangle$, then
$$A_{\lambda,j} = \sum_{\substack{\nu \vdash j \\ \lambda/\nu \text{ is a horizontal} \\ (n-j)\text{-strip}}} f^\nu.
$$
Furthermore, $A_{\lambda, j}$ satisfies the following properties:
\begin{enumerate}
    \item $A_{\lambda,j} = 0$ for $j < d$.
    \item $A_{\lambda,j} \leq \binom{j}{d} f^{\overline{\lambda}}$ for $j \geq d$, with equality holding when $d \leq j \leq \ell$.
\end{enumerate}
\end{lemma}
\begin{proof} From the \Cref{eq:p_1_j} and \Cref{lem:Pieri's_Formula}, we have
    \begin{align}
        A_{\lambda,j}&=\left\langle p_1^j h_{n-j}, s_\lambda \right\rangle\nonumber\\
        &=\sum_{\nu\vdash j}f^{\nu}\left\langle s_{\nu} \,h_{n-j},\,s_{\lambda} \right\rangle\nonumber\\
        &=\sum_{\substack{\nu\vdash j\\ \lambda/\nu \text{ is a horizontal }\\(n-j)\text{-strip}}}f^{\nu}.\label{eq:A_lambda_j_expression}
    \end{align}
    Note that $\lambda/\nu$ is a horizontal strip if and only if the interlacing inequalities
    $$\lambda_i\geq \nu_i\geq \lambda_{i+1}$$
    hold for every $i$. In particular, we have $ \nu_i\geq \lambda_{i+1}=\overline{\lambda_i}$, so $\nu\supseteq \overline{\lambda}$, and therefore $j=|\nu|\geq |\overline{\lambda}|=d$. Therefore $A_{\lambda,j}=0$ when $j<d$.

    Again, note that $\nu_i\geq \overline{\lambda_i}=\lambda_{i+1}\geq \nu_{i+1}$, therefore $\nu/\overline{\lambda}$ is a horizontal strip. Now by \Cref{lem:Pieri's_Formula}, we have 
    \begin{equation*}
        s_{\overline{\lambda}}\, h_{j-d}=\sum_{\substack{\mu\vdash j\\ \mu/\overline{\lambda} \text{ is a horizontal }\\(j-d)\text{-strip}}}s_{\mu}.
    \end{equation*}
    Therefore, Frobenius characteristic correspondence implies
    \begin{align*}
        \sum_{\substack{\mu\vdash j\\ \mu/\overline{\lambda} \text{ is a horizontal }\\(j-d)\text{-strip}}}\chi^{\mu}=\Ind_{S_d\times S_{j-d}}^{S_j}\left(\chi^{\overline{\lambda}}\otimes\chi^{(j-d)}\right)
    \end{align*}
Now, taking dimensions gives
\begin{align}
        \sum_{\substack{\mu\vdash j\\ \mu/\overline{\lambda} \text{ is a horizontal }\\(j-d)\text{-strip}}}f^{\mu}&=\dim\left(\Ind_{S_d\times S_{j-d}}^{S_j}\left(\chi^{\overline{\lambda}}\otimes\chi^{(j-d)}\right)\right)
        =\binom{j}{d}f^{\overline{\lambda}}.\label{eq:s_lambda_h_j-d_pieri_formula}
    \end{align}
Since, if $\lambda/\nu$ is a horizontal strip then $\nu/\overline{\lambda}$ is a horizontal strip. Thus, the set of $\nu\vdash j$ in the summation in \Cref{eq:A_lambda_j_expression} is a subset of the set of  $\mu\vdash j$ in the summation in \Cref{eq:s_lambda_h_j-d_pieri_formula}. Thus, $A_{\lambda,j} \leq \binom{j}{d} f^{\overline{\lambda}}$ for $j \geq d$. Again, if $d\leq j\leq \ell$, then the summation sets in \Cref{eq:A_lambda_j_expression,eq:s_lambda_h_j-d_pieri_formula} are same, therefore equality holds in this case.
\end{proof}

We next obtain another useful representation of the fixed-point transform by applying the Jacobi--Trudi identity. This gives a determinantal expression for $\Phi_\lambda(1+t)$, which can be further rewritten as a finite linear combination of the partial exponential sums $E_m(t)$.

\begin{lemma}[Jacobi--Trudi expansion]\label{lem:Jacobu-Trudi}
For $\lambda=(\lambda_1,\lambda_2,\dots,\lambda_r)\vdash n$, we have
\begin{equation*}
\Phi_\lambda(1+t)
=
\det\left[
E_{\lambda_i-i+j}(t)
\right]_{i,j=1}^{r}.
\end{equation*}
Moreover, if $l_i:=\lambda_i+r-i
\text{ and }
\delta_i:=l_1-l_i
\quad (1\leq i\leq r),
$ then
\begin{equation*}
\Phi_\lambda(1+t)
=
\sum_{i=1}^{r}
(-1)^{i+1}\,c_i\,
t^{d+\delta_i}\,
E_{\ell-\delta_i}(t),
\end{equation*}
where
\begin{equation*}
c_i:=
\frac{
\displaystyle
\prod_{\substack{p<q\\p,q\neq i}}
(l_p-l_q)
}{
\displaystyle
\prod_{p\neq i}l_p!
}.
\end{equation*}
Furthermore, $
c_i>0
\text{ for every }1\leq i\leq r,
$ and $c_1=\frac{f^{\overline{\lambda}}}{d!}$.
\end{lemma}

\begin{proof}

Define an algebra homomorphism $\theta_t:\Lambda_\Q\to \Q[t]$ as
$$\theta_t(p_1)=(1+t),\qquad \theta_t(p_k)=1\quad k\geq 2.$$
This means that $\theta_t$ preserves sums and products:
$$\theta_t(F+G)=\theta_t(F)+\theta_t(G)\quad \text{and}\quad \theta_t(F\,G)=\theta_t(F)\,\theta_t(G).$$
Therefore, for any $\rho\vdash n$, we have 
\begin{align*}
     \theta_t(p_\rho) = \theta_t(p_1)^{m_1(\rho)} \prod_{k \ge 2} \theta_t(p_k)^{m_k(\rho)} = (1+t)^{m_1(\rho)}.
\end{align*}
Now, since $$s_{\lambda}=\sum\limits_{\rho\vdash n}\frac{\chi^{\lambda}(\rho)}{z_{\rho}}p_{\rho},$$
applying $\theta_t$ both sides, we have
\begin{align}
\theta_t(s_{\lambda})
    &=\sum_{\rho\vdash n}\frac{\chi^{\lambda}(\rho)}{z_{\rho}}(1+t)^{m_1(\rho)}=\Phi_{\lambda}(1+t).\label{eqn:theta_t_schur}
\end{align}
Again, since generating function for the complete homogeneous symmetric functions is given by
$$\sum_{m\geq 0}h_mu^m=\exp\left(\sum_{k\geq 1}p_k \frac{u^k}{k}\right).$$

Applying $\theta_t$ both sides, we have
\begin{align*}
    \sum_{m\geq 0}\theta_t(h_m)u^m&=\exp\left((1+t)u+\sum_{k\geq 2}\frac{u^k}{k}\right)\\
    &=\frac{e^{tu}}{1-u}\\
    &=\sum_{m\geq 0}E_m(t)u^m.
\end{align*}
Comparing the coefficient of $u^m$ both sides, we get 
$$\theta_t(h_m)=E_m(t).$$
Now using Jacobi-Trudi determinant formula for $s_{\lambda}$ in \Cref{eqn:theta_t_schur}, we have   
\begin{align}
    \Phi_{\lambda}(1+t)
    &=\theta_t(\det\left[h_{\lambda_i -i+j}\right]_{i,j=1}^{r})=\det\left[E_{\lambda_i -i+j}(t)\right]_{i,j=1}^{r}.\nonumber
\end{align}
This proves the first part of the lemma. 

Set $k_i:=\lambda_i-i+1=\ell-\delta_i$. Using the column operation $C_j\longleftarrow C_j-C_{j-1}$ for $j=r,r-1,\ldots 2$, the $i$-th row of $M:=\left[E_{\lambda_i -i+j}(t)\right]_{i,j=1}^{r}$ becomes
$$
\left(
E_{k_i}(t),
\frac{t^{k_i+1}}{(k_i+1)!},
\ldots,
\frac{t^{k_i+r-1}}{(k_i+r-1)!}
\right).
$$

The minor corresponding to the $i$-th row and the first column is
\begin{align*}
    M_{i,1}&=t^{\,n-k_i}
\det
\left[
\frac{1}{(k_p+q)!}
\right]_{\substack{p\neq i\\1\leq q\leq r-1}}\\
&=t^{d-\delta_i}\frac{1}{\prod_{p\neq i}l_p!}\det [(l_p)_{r-q-1}]_{\substack{p\neq i\\1\leq q\leq r-1}},
\end{align*}
where $(m)_t=m(m-1)(m-2)\cdots(m-t+1)$ denotes the falling factorial.

Therefore, by \Cref{lem:General_Vandermonde}, we have 
$$M_{i,1}=t^{d-\delta_i}\frac{\prod_{\substack{p<q\\p,q\neq i}}(l_p-l_q)}{\prod_{p\neq i}l_p!}.$$

Thus, the Laplace expansion gives the second part of the theorem and $c_1=f^{\overline{\lambda}}/d!$ follows from the hook-length formula.
\end{proof}

To obtain bounds for the coefficients $c_i$ appearing in the above lemma, we will need a comparison between the dimensions of two partitions that differ by one box in the first row. The following bound provides the required comparison.

\begin{lemma}\label{lem:fTtau_bound_adding_one_box}
    let $\tau=(\tau_1,\tau_2,\ldots, \tau_r)\vdash n$ and $\tau^+=(\tau_1+1,\tau_2,\ldots, \tau_r)$ is obtained from $\tau$ by adding one box to its first row, then 
    $$\frac{f^{\tau^+}}{(n+1)f^{\tau}}\leq \frac{n}{\tau_1(\tau_1+1)}.$$
\end{lemma}
\begin{proof}
    Let $H_j$ denote the hook length of the box $(1,j)$ in the first row of $\tau$. When one box is added to the end of the first row, each $H_j$ increases by 1, while all other existing hook lengths remain unchanged. The new box itself has hook length 1. Therefore, by the hook-length formula,
    $$\frac{f^{\tau^+}}{(n+1)f^{\tau}} =  \prod_{j=1}^{\tau_1}\frac{H_j}{H_j+1}.$$
Set $\tau_{r+1}=0$. For $\tau_{s+1}<j\leq \tau_s$, the $j^{th}$ column of $\tau$ has height $s$, that is, $\tau'_j=s$. Thus
$$\prod_{j=\tau_{s+1}+1}^{\tau_s}\frac{H_j}{H_j+1}=\prod_{j=\tau_{s+1}+1}^{\tau_s}\frac{\tau_1-j+s}{\tau_1-j+s+1}=\frac{\tau_1-\tau_s+s}{\tau_1-\tau_{s+1}+s}.$$
Therefore, we have 
\begin{equation}\label{eqn:RHS_EXPRS}
    \frac{f^{\tau^+}}{(n+1)f^{\tau}} =\frac{1}{\tau_1+r}\prod_{s=2}^{r}\frac{\tau_1-\tau_s+s}{\tau_1-\tau_{s}+s-1}.
\end{equation}
We now prove the bound by induction on the number $r$ of rows. 

For $r=1$, $\tau=(\tau_1)=(n)$ and $\tau^+=(\tau_1+1)$, therefore
$$\frac{f^{(\tau_1+1)}}{(n+1)f^{(\tau_1)}}=\frac{1}{(n+1)}=\frac{n}{\tau_1(\tau_1+1)}.$$
Assume $r\geq 2$, and let $\hat{\tau}:=(\tau_1,\tau_2,\ldots, \tau_{r-1})$ and $\hat{n}:=n-\tau_r$. 
Write $$R(\tau)=\frac{f^{\tau^+}}{(|\tau|+1)f^{\tau}}.$$
By using \Cref{eqn:RHS_EXPRS}, we obtain
\begin{align*}
  \frac{R(\tau)}{R(\hat{\tau})}
  &\leq 1+\frac{\tau_r}{\hat{n}}=\,\frac{n}{\hat{n}}.
\end{align*}
Therefore, using the induction hypothesis, we have 
$$R({\tau})\leq \frac{n}{\hat{n}}\frac{\hat{n}}{\tau_1(\tau_1+1)}=\frac{{n}}{\tau_1(\tau_1+1)}.$$
\end{proof}

We now apply the preceding dimension comparison to obtain bounds on the ratios $c_i/c_1$.  The following result provides suitable bounds. 

\begin{lemma}\label{lem:bound for c_i/c_1}
    With the notation of $\Cref{lem:Jacobu-Trudi}$, for $i\geq 2$, let
    $$\tau^i=(\lambda_i,\lambda_{i+1},\ldots,\lambda_r), \qquad N_i=|\tau^i|$$
    and let $(\tau^i)^+$ be obtained from $\tau^i$ by adding one box to its first row. Then
    $$\frac{c_i}{c_1}\leq \binom{\delta_i-1}{i-2}\left(\frac{f^{(\tau^i)^+}}{(N_i+1)f^{\tau^i}}\right)^{\delta_i}.$$
    Consequently, if $0<\delta_i\leq \ell/2$
    $$\frac{c_i}{c_1}\leq \left(\frac{8n}{\ell^2}\right)^{\delta_i}.$$
    Furthermore, for every $i\geq 2,$
    $$\frac{c_i}{c_1}\leq \frac{2^{\delta_i}}{\delta_i!}.$$
\end{lemma}
\begin{proof}
 Fix $i\geq 2$, and write $\delta_i=l_1-l_i$.  From the definition of $c_i$'s, we have
 \begin{equation}\label{eq:c_i/c_1}
     \frac{c_i}{c_1}=\frac{l_i!}{l_1!}\left(\prod_{j=2}^{i-1}\frac{l_1-l_j}{l_j-l_i}\right)\left(\prod_{j=i+1}^{r}\frac{l_1-l_j}{l_i-l_j}\right)
 \end{equation}
 For $2\leq j\leq i-1$, set $a_j:=l_1-l_j$. Note that $l_1>l_j>l_i$, this implies $0<a_j<l_1-l_i=\delta_i$. As $l_j$'s are distinct integers, therefore $a_j$'s are distinct integers in $\{1,2,\ldots,\delta_i-1\}$. Now consider
 the first product
 $$\prod_{j=2}^{i-1}\frac{l_1-l_j}{l_j-l_i}=\prod_{j=2}^{i-1}\frac{a_j}{\delta_i-a_j}.$$
 The function $a\mapsto a/(\delta_i-a)$ is increasing. Therefore the product is
maximized by choosing the largest $i-2$ possible values of $a_j$, and hence
\begin{equation}\label{eq:first-product}
 \prod_{j=2}^{i-1}\frac{l_1-l_j}{l_j-l_i}
 \le \prod_{k=1}^{i-2}\frac{\delta_i-k}{\delta_i-(\delta_i-k)}
 =\binom{\delta_i-1}{i-2}.
\end{equation}

Now, for $i+1\leq j\leq r$, set $b_j:=l_i-l_j$. Using \Cref{eqn:RHS_EXPRS} of \Cref{lem:fTtau_bound_adding_one_box} for $\tau^i$, we have
\begin{align*}
    \frac{f^{(\tau^i)^+}}{(N_i+1)f^{\tau^i}}&=\frac{1}{\tau^i_1+r-i+1}\prod_{s=2}^{r-i+1}\left(1+\frac{1}{\tau^i_1-\tau^i_s+s-1}\right)\\
    &=\frac{1}{l_i+1}\prod_{j=i+1}^{r}\left(1+\frac{1}{b_j}\right).
\end{align*}

Now consider,
\begin{align}
    \frac{l_i!}{l_1!}\left(\prod_{j=i+1}^{r}\frac{l_1-l_j}{l_i-l_j}\right)&=\frac{1}{(l_i+1)(l_i+2)\cdots(l_i+\delta_i)}\prod_{j=i+1}^{r}\left(1+\frac{\delta_i}{b_j}\right)\nonumber\\
   &\le\left(\frac{f^{(\tau^i)^+}}{(N_i+1)f^{\tau^i}}\right)^{\delta_i}\label{eq:second-product}
\end{align}

From \Cref{eq:first-product,eq:second-product,eq:c_i/c_1}, we have 
$$\frac{c_i}{c_1}\leq \binom{\delta_i-1}{i-2}\left(\frac{f^{(\tau^i)^+}}{(N_i+1)f^{\tau^i}}\right)^{\delta_i}.$$

Now, suppose that $0<\delta_i\leq \ell/2$. Since $\lambda_i=\ell+i-1-\delta_i\geq \ell-\delta_i\geq \ell/2$, therefore by \Cref{lem:fTtau_bound_adding_one_box}, we have 
$$\frac{f^{(\tau^i)^+}}{(N_i+1)f^{\tau^i}}\leq \frac{N_i}{\lambda_i^2}\leq \frac{4n}{\ell^2}.$$
Therefore, for $0<\delta_i\leq \ell/2$, we have 
$$\frac{c_i}{c_1}\leq 2^{\delta_i-1}\left(\frac{4n}{\ell^2}\right)^{\delta_i}\leq \left(\frac{8n}{\ell^2}\right)^{\delta_i}.$$
For the universal estimate, the numbers $b_j=l_i-l_j$, $j>i$, are distinct
positive integers not exceeding $l_i$. Hence
$$
 \prod_{j=i+1}^r\left(1+\frac{\delta_i}{b_j}\right)
 \le\prod_{b=1}^{l_i}\left(1+\frac{\delta_i}{b}\right)
 =\binom{l_1}{\delta_i}.
$$
Therefore, 
$$
 \frac{c_i}{c_1}
 \le2^{\delta_i-1}\frac{l_i!}{l_1!}\binom{l_1}{\delta_i}
 \le\frac{2^{\delta_i}}{\delta_i!},
$$
for every $i\geq2$.
 \end{proof}

\subsection{An Optimization Bound}

The following theorem gives the key bound that will be used in the proof of our main result for even $n$.

\begin{theorem}\label{thm:z_upper_bound}
Let 
$
\mathcal{O}_n = \{ \rho \vdash n : m_1(\rho) \text{ is odd} \} 
$ be a subset of partitions with odd number of fixed points.
Suppose $\{x_\lambda\}$ is a set of non-negative weights indexed by partitions $\lambda \vdash n$ with $\lambda \neq (n)$, such that
$$
\sum_{\substack{\lambda \vdash n \\ \lambda \neq (n)}} x_\lambda = 1.
$$
Furthermore, assume that for every $\rho \in \mathcal{O}_n$, a real number $z$ satisfies the bound
$$
z \le -\sum_{\substack{\lambda \vdash n \\ \lambda \neq (n)}} x_\lambda \frac{\chi^\lambda(\rho)}{f^\lambda}.
$$
Then, there exist absolute constants $C>0$ and $n_0$ such that for all even $n \ge n_0$,
$$
z \le \frac{\exp{\left(\frac{n}{2}+Cn^{3/4}\log n\right)}}{(n-1)!!}.
$$
Consequently, for sufficiently large $n$
$$
z \le \frac{e^{\frac{n}{2}+o(n)}}{(n-1)!!}.
$$
\end{theorem}
\begin{proof}
    Let \(\pi\) be any probability distribution supported on \(\OO_n\).
Averaging the given inequalities against \(\pi\) gives
\begin{align}
z
&\le
-\sum_{\rho\in\OO_n}\pi(\rho)
\sum_{\substack{\lambda\vdash n\\ \lambda\ne(n)}}
x_\lambda
\frac{\chi^\lambda(\rho)}{f^\lambda}
\nonumber\\
&\le
\max_{\substack{\lambda\vdash n\\ \lambda\ne(n)}}
\left(
-\sum_{\rho\in\OO_n}
\pi(\rho)\frac{\chi^\lambda(\rho)}{f^\lambda}
\right)= \max_{\substack{\lambda\vdash n\\ \lambda\ne(n)}}\left(-\E_{\pi}\left[\frac{\chi^{\lambda}}{f^\lambda}\right]\right).\label{eq:z_upper_exprssn}
\end{align}
Now, we define a probability measure on $\OO_n$ as follows:

Fix $a=\sqrt{n}$, $X=a+1$ and $Y=a-1$. Now, we define 
$$Z_a:=\sum_{\rho\in \OO_n}\frac{a^{m_1(\rho)}}{z_{\rho}}$$
and the probability measure 
$$\pi_{a}(\rho):=\frac{a^{m_1(\rho)}}{z_{\rho}Z_{a}},\qquad \text{ for every  }\rho\in \OO_n.$$

By the definition of $\Phi_\lambda$, we have
$$\sum_{\rho\in\OO_n}
\frac{a^{m_1(\rho)}}{z_\rho}\chi^\lambda(\rho)
=
\frac{\Phi_\lambda(a)-\Phi_\lambda(-a)}{2},$$
and consequently,
\begin{equation}
\label{eq:Fourier-coefficient}
\E_{\pi_a}
\left[
\frac{\chi^\lambda}{f^\lambda}
\right]
=
\frac{\Phi_\lambda(a)-\Phi_\lambda(-a)}
{2Z_af^\lambda}.\nonumber
\end{equation}

Define $q_{\lambda}:=-\E_{\pi_a}
\left[
\frac{\chi^\lambda}{f^\lambda}
\right]$, therefore by \Cref{eq:z_upper_exprssn}, we have 
$$z\leq \max_{\substack{\lambda\vdash n\\ \lambda\ne(n)}}\;q_{\lambda}.$$

\begin{claim}\label{claim:Z_a Lower Bound}
    For sufficiently large $n$, $Z_a\geq \frac{1}{4}e^{a-1}.$
\end{claim}
\textit{Proof of $\cref{claim:Z_a Lower Bound}$}: 
First note that 
\begin{align*}
    2Z_a&=\Phi_{(n)}(a)-\Phi_{(n)}(-a)\\
        &= E_n(Y)-E_n(-X).
\end{align*}
Since $ E_n(Y)=e^Y(1+o(1))$ and $E_n(-X)=e^{-X}(1+o(1))
$, we have 
\begin{align*}
    2Z_a=(e^Y - e^{-X})(1+o(1))= e^Y(1+o(1)).
\end{align*}
Therefore $Z_a\geq \frac{1}{4}e^{a-1}$ for sufficiently large $n$.

\begin{claim}\label{claim:bound on lambda_1}
    There are absolute constants $C_0>0$ and $n_1$ such that for any $n\geq n_1$, $$q_{\lambda}>0 \implies\lambda_1<C_0\,n^{3/4}.$$
\end{claim}
\textit{Proof of $\cref{claim:bound on lambda_1}$}: Let $\ell:=\lambda_1$, $d:=n-\ell$ and suppose that $$\ell\geq C_0\,n^{3/4}.$$

Recall that, 
$$q_{\lambda}=\frac{\Phi_\lambda(-a)-\Phi_\lambda(a)}
{2Z_af^\lambda}.$$
Since $Z_a>0$ and $f^\lambda>0$, therefore
$$q_{\lambda}<0\iff \Phi_\lambda(-a)<\Phi_\lambda(a).$$
Now, from \Cref{lem:Jacobu-Trudi}, we have 
\begin{align*}
    \Phi_\lambda(a)
&=
\sum_{i=1}^{r}
(-1)^{i+1}\,c_i\,
Y^{d+\delta_i}\,
E_{\ell-\delta_i}(Y)\\
&=c_1Y^dE_{\ell}(Y)\left(1+\sum_{i=2}^r\frac{(-1)^{i+1}c_iY^{d+\delta_i}E_{\ell-\delta_i}(Y)}{c_1Y^{d}E_{\ell}(Y)}\right);
\quad\text{and} \\
\Phi_\lambda(-a)
&=
\sum_{i=1}^{r}
(-1)^{i+1+d+\delta_i}\,c_i\,
X^{d+\delta_i}\,
E_{\ell-\delta_i}(-X)\\
&=(-1)^d c_1X^dE_{\ell}(-X)\left(1+\sum_{i=2}^r\frac{(-1)^{i+1+\delta_i}c_iX^{d+\delta_i}E_{\ell-\delta_i}(-X)}{c_1X^{d}E_{\ell}(-X)}\right).
\end{align*}

Define,
\begin{align*}
    T_{a}&:=\sum_{i=2}^r\frac{(-1)^{i+1}c_iY^{d+\delta_i}E_{\ell-\delta_i}(Y)}{c_1Y^{d}E_{\ell}(Y)};\qquad \text{and}\\
    T_{-a}&:=\sum_{i=2}^r\frac{(-1)^{i+1+\delta_i}c_iX^{d+\delta_i}E_{\ell-\delta_i}(-X)}{c_1X^{d}E_{\ell}(-X)}.
\end{align*}
We want to estimate $|T_a|$ and $|T_{-a}|$. First note that, for $0\leq h\leq \ell/2$, we have $\ell-h\geq \ell/2$, and consequently $E_{\ell-h}
(-X)=e^{-X}(1+o(1))$. Therefore
$$\left|\frac{E_{\ell-h}(-X)}{E_{\ell}(-X)}\right|=1+o(1)\leq 2\qquad \text{ for }0\leq h\leq \ell/2.$$
Since $\delta_i$ is strictly increasing, there is at most one index $i$ for each value of $\delta_i$. Therefore, by using the monotonicity of $E_m(Y)$ and the bound for $c_i/c_1$ from \Cref{lem:bound for c_i/c_1}, we have 
\begin{align*}
    |T_a|&\leq \sum_{0<\delta_i\leq \ell/2} \frac{c_iY^{d+\delta_i}E_{\ell-\delta_i}(Y)}{c_1Y^{d}E_{\ell}(Y)} + \sum_{\delta_i>\ell/2} \frac{c_iY^{d+\delta_i}E_{\ell-\delta_i}(Y)}{c_1Y^{d}E_{\ell}(Y)}\\
    &\leq \sum_{h\geq 1}\left(\frac{8nY}{\ell^2}\right)^h+\sum_{h>\ell/2}\frac{(2Y)^h}{h!}.
\end{align*}
Since $\ell\geq C_0n^{3/4}$, we can choose $C_0$ sufficiently large such that $8nY/\ell^2$ becomes less than 1 and the first sum converges as a geometric series to a small number. Also note that since $h\geq \ell/2$, the second sum is $o(1)$. Therefore, for sufficiently large $n$, we have $|T_{a}|\leq 1/4$.

Similarly, note that $|E_m(-X)|\leq e^X$ for $m\geq 0$, and $E_\ell(-X)\geq \frac{1}{2}e^{-X}$ for all sufficiently large $n$. Therefore, by \Cref{lem:bound for c_i/c_1}, we have
\begin{align*}
    |T_{-a}|&\leq \sum_{0<\delta_i\leq \ell/2} \frac{c_iX^{d+\delta_i}|E_{\ell-\delta_i}(-X)|}{c_1X^{d}E_{\ell}(-X)} + \sum_{\delta_i>\ell/2} \frac{c_iX^{d+\delta_i}|E_{\ell-\delta_i}(-X)|}{c_1X^{d}E_{\ell}(-X)}\\
    &\leq 2\sum_{h\geq 1}\left(\frac{8nX}{\ell^2}\right)^h+2e^{2X}\sum_{h>\ell/2}\frac{(2X)^h}{h!}.
\end{align*}
Put $K:=\lfloor\ell/2\rfloor+1$. For sufficiently large $n$, $K\geq 4X$, thus
\begin{align*}
    2e^{2X}\sum_{h>\ell/2}\frac{(2X)^h}{h!}\leq 4e^{2X}\frac{(2X)^K}{K!}\leq 4e^{2X}\left(\frac{2eX}{K}\right)^K\leq 4\exp\left(2X-\frac{C_0}{16}n^{3/4}\log n\right)=o(1).
\end{align*}
Therefore, for sufficiently large $n$, we have $|T_{-a}|\leq 1/4$.

Therefore, for sufficiently large $n$,
\begin{align}
    \Phi_{\lambda}(a)&=c_1Y^dE_{\ell}(Y)\left(1+\epsilon^+\right)\quad \text{and}\label{eq:phi_a}\\
    \Phi_{\lambda}(-a)&=(-1)^{d}c_1X^dE_{\ell}(-X)\left(1+\epsilon^-\right). \label{eq:phi_ma}  
\end{align}
Where $|\epsilon^+|\leq 1/4$ and $|\epsilon^-|\leq 1/4$.

If $d$ is odd, then by \Cref{eq:phi_a,eq:phi_ma}, $\Phi_{\lambda}(a)>0$ and $\Phi_{\lambda}(-a)<0$, therefore $q_{\lambda}<0$. Suppose now that $d$ is even. Again by  \Cref{eq:phi_a,eq:phi_ma}, we have
\begin{align*}
    \frac{\Phi_{\lambda}(-a)}{\Phi_{\lambda}(a)}&\leq \frac{5}{3}\left(\frac{X}{Y}\right)^d\frac{E_\ell(-X)}{E_\ell(Y)}\\
    &=\frac{5}{3}(1+o(1))\exp\left(d\log \frac{X}{Y}-2a\right)\\
    &\leq \frac{5}{3}(1+o(1))\exp\left(-\frac{2(\ell-a)}{a-1}\right)=o(1).
\end{align*}
Therefore, for sufficiently large $n$, $q_{\lambda}<0$. We have shown that every \(\lambda\) satisfying
\(\lambda_1\ge C_0n^{3/4}\) has \(q_\lambda\le0\), which proves the \Cref{claim:bound on lambda_1}.

Since, we are interested in the upper bound, we estimate $q_\lambda$ when it is positive. Therefore, by \Cref{claim:bound on lambda_1}, we can assume that $d=n-\ell=n-O(n^{3/4})$. Since $f^{\overline{\lambda}}\leq f^{\lambda}$ and $\Phi_{\lambda}(t)\geq 0$ when $t>1$, therefore, by the \Cref{lem:A_lambda_j bound,lem:Phi_lambda_series_expansion,claim:Z_a Lower Bound} and the  definition of $q_\lambda$, we have 
\begin{align*}
    q_\lambda&\leq \frac{\Phi_{\lambda}(-a)}{2Z_a f^{\lambda}} \leq \frac{f^{\overline{\lambda}}X^de^X/d!}{(e^{a-1}/2)f^{\lambda}} \leq 2e^2\frac{X^d}{d!}.
\end{align*}
Now recall that $X=\sqrt{n}+1$, and $d=n-O(n^{3/4})$. Substituting in the above inequality we get the result.
\end{proof}

\section{Even-Intersecting Families of Permutations}
\label{sec:Even_intersecting_family}
\subsection{The Case of Even  \texorpdfstring{$n$}{n}}
\label{subsec:Even_intersecting_family_even_n}

Throughout this subsection we consider $n$ to be even. Let
$\mathcal{F}\subseteq S_n$ be an even-intersecting family, that is, $|\Fix(\sigma \,\tau^{-1})|$ is even for  $\sigma,\tau\in \mathcal{F}$. Now, consider the following subset of $S_n$
$$
O_n
:=
\{\sigma\in S_n: |\Fix(\sigma)|\text{ is odd}\}.
$$
 Now, we consider the normal Cayley graph
$$
\Gamma:=\operatorname{Cay}(S_n,O_n).
$$ Thus,
$$
\sigma\sim\tau
\quad\Longleftrightarrow\quad
|\Fix(\sigma\tau^{-1})|\text{ is odd}.
$$
Consequently, a family of permutations is even-intersecting if and only if it
is an independent set in $\Gamma$. Therefore,
$$
M(n)=\alpha(\Gamma).
$$

Hence, the problem is reduced to finding an upper bound for the independence number of
$\Gamma$. We want to use weighted Hoffman bound to get an upper bound, and for that we seek a weighted adjacency
matrix for which the least eigenvalue is as large as possible while fixing the largest eigenvalue to be 1.

\begin{proof}[Proof of $\Cref{thm:main_result}$]
Let
$$
\mathcal{O}_n
:=
\{\rho\vdash n:m_1(\rho)\text{ is odd}\},
$$
and, for $\rho\in\mathcal{O}_n$, let $X_\rho$ denote the conjugacy class of
$S_n$ of cycle type $\rho$. 
Let $A_\rho$ be the adjacency matrix of $\Cay(S_n,X_\rho)$, for each $\rho\in \OO_n$ and consider the weighted adjacency matrix
$$
A
:=
\sum_{\rho\in\mathcal{O}_n}c_\rho A_\rho.
$$
By \Cref{lem:Weighted_Adjacency_eigenvlues}, for every
$\lambda\vdash n$, the corresponding eigenvalue is
$$
\theta_\lambda
=
\sum_{\rho\in\mathcal{O}_n}
c_\rho |X_\rho|\frac{\chi^\lambda(\rho)}{f^\lambda}.
$$

We want to find the weights $c_\rho$, $\rho\in \OO_n$ such that $\theta_{(n)}=1$ is the largest eigenvalue and the smallest eigenvalue is as large as possible. This leads to the following linear program:

\noindent\textbf{Primal LP:}

\noindent Minimize $y$ subject to:
\begin{eqnarray*}
    \sum_{\rho \in \OO} \left( \frac{\chi^{\lambda}(\rho) }{f^{\lambda}} \right) \omega_{\rho} + y &\ge& 0, \quad \text{for all } \lambda \neq (n),\\
    \sum_{\rho \in \OO} \omega_{\rho} &=& 1,\\
    \omega_\rho&\geq& 0, \qquad\rho\in\OO.
\end{eqnarray*}

In the above linear program, $\omega_\rho=c_\rho|X_\rho|$. We restrict to nonnegative weights $\omega_\rho$, and hence to $c_\rho\geq 0$. Although signed weights are allowed in the weighted Hoffman framework, this restriction merely narrows the class of admissible matrices and still yields a valid upper bound for $M(n)$.

Feasibility of the primal LP is easy: Take $\omega_{(n-1,1)}=1$ and all other $\omega_\rho$ zero, and $y=2$. Determining the optimal weights
$\omega_\rho$ for the primal LP directly appears difficult, therefore we consider the dual of this LP as follows.

\noindent\textbf{Dual LP:}

\noindent Maximize $z$ subject to:
\begin{eqnarray*}
    \sum_{\substack{\lambda\vdash n\\\lambda\neq(n)}} \left( \frac{\chi^{\lambda}(\rho) }{f^{\lambda}} \right) x_{\lambda} + z &\leq& 0, \quad \text{for all } \rho \in \OO,\\
    \sum_{\substack{\lambda\vdash n\\\lambda\neq(n)}} x_{\lambda} &=& 1,\\
    x_{\lambda}&\geq& 0, \qquad \lambda\vdash n,\; \lambda\neq (n).
\end{eqnarray*}
First, we construct a feasible solution for the dual. 
Take  
\begin{align*}
    z&=\frac{1}{(n-1)!!-1},\\
    x_{2\nu}&=\frac{f^{2\nu}}{(n-1)!!-1}\qquad \text{for all }\nu\vdash n/2,\;\nu\neq (n/2),
\end{align*}
and put all other $x_\lambda=0$. It is easy to verify that this is a feasible solution for every $n\geq 6$. In fact, we have verified in the SageMath for $n=6$ to $16$ that this is indeed an optimal solution. Assuming this as an optimal solution, we get $\theta_{\min}=-1/((n-1)!!-1)$, and therefore, Hoffman's ratio bound will give
$$M(n)\leq \frac{n!}{1+(n-1)!!-1}=n!!,$$
which will answer the Cameron--Deza--Frankl problem, but proving the optimality is difficult.

Let $z_n$ be the optimal solution for the dual. The purpose of the
character optimization problem studied in the previous section was precisely
to obtain a sufficiently strong upper bound for $z_n$. By
\Cref{thm:z_upper_bound}, there exist absolute constants
$C>0$ and $n_0$ such that, for every even $n\geq n_0$, we have
$$
z_n
\leq
\frac{
\exp\left(
\frac{n}{2}
+
C n^{3/4}\log n
\right)
}{
(n-1)!!
}.
$$
Let $y_n$ be the optimal solution of the primal, then by strong duality, we have
$$
y_n=z_n,
$$
and hence there exists a choice of weights $\omega_\rho$ for which the least
eigenvalue of the corresponding weighted adjacency matrix satisfies
$$\theta_{\min}\geq-\frac{\exp\left(\frac{n}{2}+C n^{3/4}\log n\right)}{(n-1)!!}.$$

Applying the weighted Hoffman bound gives
\begin{align*}
    M(n)&\leq\frac{n!}{1-\frac{1}{\theta_{\min}}}\\
    &\leq n!!\exp\left(\frac{n}{2}+C n^{3/4}\log n\right).
\end{align*}

Equivalently, for sufficiently large even $n$
$$
M(n)\leq e^{\, n/2+o(n)}\,n!!.
$$

For the lower bound, let $n=2m$. In its natural action on $[n]$, the wreath product $H=S_2\wr S_m\le S_n$ consists of all permutations preserving the partition into pairs
$$B_i=\{2i-1,2i\},\qquad 1\le i\le m.$$
Thus an element of $H$ may permute the pairs and independently swap the two points within each pair. Therefore it gives an even-intersecting family of size $|S_2\wr S_{n/2}|=n!!$.
\end{proof}

\subsection{The case of  Odd  \texorpdfstring{$n$}{n}}
\label{subsec:Even_intersecting_family_odd_n}

We first give a simple construction that works for every odd $n$. The idea is to take two odd cycles whose supports intersect in exactly one point and consider products of their powers.

\textbf{Construction 1:} 
Let $r$ and $s$ be odd numbers such that $n=r+s-1$. Define the following two cycles:
$$\alpha:=(1, 2, \ldots, r)\qquad\text{ and }\qquad\beta:=(r,r+1,\ldots,n).$$
Now, consider the family
$$\mathcal{F}:=\{\pi_{i,j}:=\beta^j\alpha^i: 0\leq i<r, 0\leq j<s\}\subseteq S_n.$$
Then, $\mathcal{F}$ is an even-intersecting family of size $rs$.

\begin{proof}
We compose permutations from right to left. Now, for any $i,i',j,j'\in [n]$, we want to know the intersection size $$|\pi_{i,j}\cap \pi_{i',j'}|=|\{x\in[n]: \pi_{i,j}(x)=\pi_{i',j'}(x)\}|.$$

\textbf{Case 1:} $i= i'$ and $j\neq j'$.

Note that, for $x\in [r+1, n]$, $\pi_{i,j}(x)=\beta^j(\alpha^i(x))=\beta^j(x)$, therefore $\pi_{i,j}(x)\neq \pi_{i,j'}(x)$ for every $x\in [r+1,n]$. Now suppose $\alpha^i(y)=r$, then for any $x\in [r]\setminus\{y\}$, we have $\pi_{i,j}(x)= \pi_{i,j'}(x)$ and $\pi_{i,j}(y)=\beta^j(r)\neq\beta^{j'}(r)= \pi_{i,j'}(y)$. Therefore $|\pi_{i,j}\cap \pi_{i,j'}|=r-1$.

\textbf{Case 2:} $i\neq i'$ and $j= j'$.

Note that $\pi_{i,j}(x)=\pi_{i',j}(x)$ if and only if $\alpha^i(x)=\alpha^{i'}(x)$. Also, $\alpha^i(x)=\alpha^{i'}(x)$ only for $x\in [r+1, n]$. Therefore $\pi_{i,j}(x)=\pi_{i',j}(x)$ only for $x\in [r+1, n]$. Hence,  $|\pi_{i,j}\cap \pi_{i',j}|=s-1$.

\textbf{Case 3:} $i\neq i'$ and $j\neq j'$.

Again, if $x\in[r+1,n]$, then $\pi_{i,j}(x)=\beta^j(x)\neq \beta^{j'}(x)=\pi_{i',j'}(x)$. For $x\in [r]$, if $\alpha^i(x)\neq r$ then $\beta^j(\alpha^i(x))=\alpha^i(x)$. Also, at most one of $\alpha^i(x)$ or $\alpha^{i'}(x)$ can be $r$, therefore  $\pi_{i,j}(x)\neq\pi_{i',j'}(x)$ for all $x\in [r]$. Hence, $|\pi_{i,j}\cap \pi_{i',j'}|=0$.

Since, $r$ and $s$ are odd, $\mathcal{F}$ is an even-intersecting family of size $rs$. 
\end{proof}

\begin{proof}[Proof of $\Cref{Thm:loose lower bound}$]
    If $n=4k+3$, then take $r=2k+1$ and $s=2k+3$, then we have an even-intersecting family of size $4k^2+8k+3=(n+1)^2/4-1$. Similarly, if $n=4k+1$, taking $r=s=2k+1$, we have an even-intersecting family of size $(2k+1)^2=(n+1)^2/4$.
\end{proof}

Since $$rs\leq \frac{(r+s)^2}{4}=\frac{(n+1)^2}{4},$$ varying the cycle lengths in Construction 1 cannot produce a family larger than $(n+1)^2/4$. In that construction, the supports of $\alpha$ and $\beta$ intersect only at the point $r$. It is therefore natural to ask whether allowing the cycles to share more points can yield larger families while preserving the even-intersection property. The following construction answers this question affirmatively.

\textbf{Construction 2:} Let $v\geq 3$ be odd integer and let $D=\{d_1<d_2<\cdots<d_k\}\subseteq [v]$, where $k$ is even. Put $n:=v+k$ and choose the  cycles in $S_n$ as $$\alpha:=(1,2,\ldots,v),\qquad\text{and}\qquad \beta:=(b_1,b_2,\ldots,b_{v}),$$
where $b_t=t$ if $t\notin D$ and $b_t=v+j$ if $t=d_j$. Now define $\alpha^h(D):=\{\alpha^h(d):d\in D\}$. Suppose that $|D\cap \alpha^h(D)|$ is odd for every $1\leq h\leq v-1$. Now define the family
$$\mathcal{F}_D:=\{\pi_{i,j}:=\beta^j\alpha^i: 0\leq i,j<v\}\subseteq S_n.$$
Then, $\mathcal{F}_D$ is an even-intersecting family of size $v^2$.
\begin{proof}
    Take any two elements $\pi_{i,j},\,\pi_{i',j'}\in \mathcal{F}_D$ and set $a=i'-i$ and $b=j'-j$. Thus, we can write 
    $$|\pi_{i,j}\cap \pi_{i',j'}|=|\Fix\left(\beta^b\alpha^a\right)|.$$
    If any of $a$ or $b$ is zero then $|\pi_{i,j}\cap \pi_{i',j'}|$ is even as $|\Fix\left(\alpha^a\right)|=|\Fix\left(\beta^b\right)|=k$ for every $1\leq a,b\leq v-1$. Therefore, we now assume that both $a$ and $b$ are nonzero. If $x\notin C:=[v]\setminus D$, then $x$ is fixed by exactly one of $\alpha$ or $\beta$ and not  fixed by the other one. Therefore $\beta^b\alpha^a(x)\neq x$. Hence, the fixed points of $\beta^b\alpha^a$ are in $C$. 
    Now suppose $x\in C$ such that $\beta^b\alpha^a(x)= x$, then we must have $\alpha^a(x)\in C$ and $\alpha^{a+b}(x)=x$. Also, note that if $\alpha^{a+b}\neq \operatorname{Id}$, then it fixes no point from $C$. Therefore, consider $\alpha^{a+b}=\operatorname{Id}$. This gives
    \begin{align*}
        |\pi_{i,j}\cap \pi_{i',j'}|&=|\{x\in C:\alpha^a(x)\in C\}|\\
        &=|C\cap \alpha^{-a}(C)|\\
        &=v-2k+|D\cap \alpha^{-a}(D)|.
    \end{align*}
    Since $|D\cap \alpha^{-a}(D)|=|D\cap \alpha^{h}(D)|$ for some $1\leq h\leq v-1$ when $a$ is nonzero, therefore $|\pi_{i,j}\cap \pi_{i',j'}|$ is even. Thus $\mathcal{F}_D$ is an even-intersecting family.
\end{proof}
Note that if we take $D=\{1,2,\ldots,v-1\}$, then construction 2 becomes construction 1 with $r=s=v$. Now the question is whether we can choose a set $D$, which gives larger intersecting family? First, consider the following example.
\begin{example}
    Take $v=13$, therefore $\alpha=(1,2,\ldots,13)$, and $D=\{1,2,4,10\}$. One can easily verify that 
    $$|D\cap \alpha^h(D)|=1\qquad \text{for all }1\leq h\leq 12.$$
    Therefore, the construction 2 gives an even-intersecting family $\mathcal{F}_D$ of size $v^2=169$ inside $S_{17}$. For comparison, the construction 1 gives family of size $81$ for $n=17$.
\end{example}

The following lemma guaranties the existence of such sets (called Singer difference sets). For the proof one can see the book by Lint and Wilson \cite{van2001course}.

\begin{lemma}[Singer \cite{singer1938theorem}, Lint and Wilson \cite{van2001course}]\label{lem:SingerSet}
    Let $q$ be a prime power, put $v=q^2+q+1$, and let $\alpha=(1,2,\ldots,v)$. There exists a subset $D\subseteq [v]$ of size $q+1$ such that 
    $$|D\cap \alpha^h(D)|=1\qquad \text{for all }1\leq h\leq v-1.$$
\end{lemma}
\begin{proof}[Proof of $\Cref{Thm:tight lower bound}$]

Take $q$ to be odd prime power, there exists a subset $D$ satisfying \Cref{lem:SingerSet}. Therefore, the construction 2 gives an even-intersecting family of size $(q^2+q+1)^2$ inside $S_{q^2+2q+2}$.  Hence $$M(q^2+2q+2)\geq (q^2+q+1)^2.$$
Taking $n=q^2+2q+2$, we have 
$M(n)\geq (n-\sqrt{n-1})^2=n^2-O(n^{3/2}).$
\end{proof}

\begin{acknowledgement}
The authors thank Hiranya Kishore Dey for bringing this problem to their attention and for many fruitful discussions, particularly regarding the construction for odd $n$.
 This research work has no associated data. 
 The work of the author Anirban Banerjee is supported by the Department of Biotechnology, India (No. BT/PR40182/BTIS/137/88/2025).
  The work of the author Abisek Dewan is supported by University Grants Commission, India (Beneficiary Code/Flag: BWBDA00147662 U).
 The work of the author Rajiv Mishra is supported by Council of Scientific $\And$ Industrial Research, India(File number: 09/921(0347)/2021-EMR-I).   
\end{acknowledgement}

\bibliographystyle{siam}
 	\bibliography{evenbib}
    \nocite{*}

\appendix
\section{Asymptotic size of the Cameron--Deza--Frankl character sum}
\label{app:CDF-asymptotic}

\begin{lemma}\label{lem:CDF-sum-asymptotic}
For even $n$, let
$$
S(n):=\sum_{\substack{\lambda\vdash n\\\lambda_1\geq n/2}}
(f^\lambda)^2.
$$
For every integer $m\geq 1$,
$$
\frac{1}{(m+1)^2}\binom{2m}{m}^{2}m!
\leq S(2m)
\leq \frac{1}{2}\binom{2m}{m}^{2}m!.
$$
Consequently, as $n\to\infty$ through even integers,
$$
S(n)=8^{n/2+O(\log n)}\,n!!.
$$
\end{lemma}

\begin{proof}
\emph{Upper bound.}
Every partition of $2m$ with first row at least $m$ can be written
uniquely as $(2m-k,\mu)$, where $0\leq k\leq m$ and $\mu\vdash k$.
In a standard tableau of this shape, the entry $1$ belongs to the first
row. Thus the $k$ entries below the first row can be chosen in at most
$\binom{2m-1}{k}$ ways. Therefore
$$
f^{(2m-k,\mu)}\leq \binom{2m-1}{k}f^\mu.
$$
By the definition of $S(2m)$, we have
\begin{align*}
    S(2m)&\leq \sum_{k=0}^{m}\sum_{\mu\vdash k}\left(\binom{2m-1}{k}f^{\mu}\right)^2\\
    &\leq \sum_{k=0}^{m}\binom{2m-1}{k}^{2}k!\\
    &\leq \frac{1}{4}\binom{2m}{m}^{2}\sum_{k=0}^{m}k!\\
    &\leq  \frac{1}{2}\binom{2m}{m}^{2}m!.
\end{align*}

\medskip
\noindent\emph{Lower bound.}
We first count standard tableaux of shape $(m,m)$. Using hook-length formula, we get 
$$
f^{(m,m)}
=\frac{(2m)!}{(m+1)!m!}
=\frac{1}{m+1}\binom{2m}{m}.
$$

For every $\mu\vdash m$, we next show that
$$
f^{(m,\mu)}\geq f^{(m,m)}f^\mu.
$$
Take a standard tableau of shape $(m,m)$ whose rows are
$$
a_1<\cdots<a_m
\qquad\text{and}\qquad
b_1<\cdots<b_m.
$$
Take independently a standard tableau $T$ of shape $\mu$, replace each
entry $k$ of $T$ by $b_k$, and place the row $a_1,\ldots,a_m$ above it.
The resulting filling has shape $(m,\mu)$. The row and column
inequalities within the lower rows are preserved. If the entry of $T$
in its first row and column $j$ is $k$, then $k\geq j$, so
$$
a_j<b_j\leq b_k.
$$
Thus the inequalities between the new first and second rows hold as
well, and the resulting tableau is standard.

The construction is injective: the new first row recovers the $a_j$.
The remaining entries, in increasing order, recover the $b_j$, and
replacing each $b_k$ in the lower rows by $k$ recovers $T$.
Therefore, 
$$
\begin{aligned}
S(2m)
&\geq\sum_{\mu\vdash m}\bigl(f^{(m,\mu)}\bigr)^2\\
&\geq\frac{1}{(m+1)^2}\binom{2m}{m}^{2}
\sum_{\mu\vdash m}(f^\mu)^2\\
&=\frac{1}{(m+1)^2}\binom{2m}{m}^{2}m!.
\end{aligned}
$$
This proves the two-sided bound.

\medskip
\noindent\emph{Asymptotic conclusion.}
The central binomial coefficient is the largest coefficient in the
binomial expansion of $(1+1)^{2m}$. Thus
$$
\frac{4^m}{2m+1}\leq\binom{2m}{m}\leq 4^m.
$$
Using the two-sided bound and $(2m)!!=2^m m!$, we obtain
$$
\frac{8^m(2m)!!}{(m+1)^2(2m+1)^2}
\leq S(2m)\leq\frac{1}{2} 8^m(2m)!!.
$$
Taking logarithms gives
$$
\log\!\left(\frac{S(2m)}{8^m(2m)!!}\right)=O(\log m).
$$
Therefore $S(2m)=8^{m+O(\log m)}(2m)!!$, as required.
\end{proof}

\end{document}